\documentclass[a4paper,10pt]{amsart}
\usepackage[T1]{fontenc}
\usepackage{lmodern}
\usepackage[all]{xy}
\usepackage{amssymb}
\usepackage{enumitem,nicefrac,mathtools}
\usepackage[british]{babel}

\usepackage{xcolor}

\newcommand*{\MRref}[2]{\linebreak[0] \href{http://www.ams.org/mathscinet-getitem?mr=#1}{MR \textbf{#1}}}

\usepackage[pdftitle={Kirchberg X-algebras with real rank zero and intermediate cancellation},
  pdfauthor={Rasmus Bentmann},
  pdfsubject={Mathematics; MSC 18G20, 19K35, 46L80}
]{hyperref}
\usepackage[lite]{amsrefs}
\usepackage{microtype}

\usepackage{graphicx}
\usepackage{tikz}
\usetikzlibrary{matrix,arrows,calc}
\tikzset{cd/.style=matrix of math nodes,row sep=2em,column sep=2em, text height=1.5ex, text depth=0.5ex}
\tikzset{cdar/.style=->,auto}
\tikzset{mid/.style={anchor=mid}} 
\tikzset{dar/.style={double,double equal sign distance,-implies}}
\tikzset{narrowfill/.style={inner sep=1pt, fill=white}}

\usepackage{array}
\usepackage[labelformat=simple]{subfig}

\DeclarePairedDelimiter{\bigbr}{\bigl(}{\bigr)}

\DeclareMathOperator{\Hom}{Hom}
\DeclareMathOperator{\Prim}{Prim}
\DeclareMathOperator{\Ext}{Ext}

\DeclareMathOperator{\rank}{rank}

\newcommand*{\KK}{\textup{KK}}
\newcommand*{\K}{\textup{K}}

\newcommand*{\XK}{\textup{XK}}
\newcommand*{\OK}{\mathbb{O}\textup{K}}
\newcommand*{\Res}{\textup{Res}}
\newcommand*{\Colim}{\textup{Colim}}

\newcommand*{\FK}{\textup{FK}}

\newcommand{\Catgunnar}{\mathcal{R}}

\newcommand*{\Ideal}{\mathfrak I}
\newcommand*{\Cstarcat}{\mathfrak{C^*alg}}

\newcommand*{\KKcat}{\mathfrak{KK}}

\newcommand*{\GCAb}{\mathfrak{Ab}_{\mathrm{c}}^{\mathbb{Z}/2}}

\newcommand*{\GCMod}[1]{\mathfrak{Mod}(#1)}  
\newcommand*{\AbelianCat}{\mathcal{A}}

\newcommand*{\op}{\mathrm{op}}

\newcommand*{\C}{\mathbb{C}}

\newcommand*{\Z}{\mathbb{Z}}
\newcommand*{\N}{\mathbb{N}}
\newcommand*{\Open}{\mathbb{O}}     
\newcommand*{\Compacts}{\mathbb{K}}

\newcommand*{\nb}{\nobreakdash}  

\newcommand*{\Bootstrap}{{\mathcal B}}
\newcommand*{\Cuntz}{{\mathcal O}}

\newcommand*{\Cstar}{\texorpdfstring{$C^*$\nb-}{C*-}}
\newcommand*{\Star}{\texorpdfstring{$^*$\nb-}{*-}}

\newcommand*{\defeq}{\mathrel{\vcentcolon=}}

\newcommand*{\inOb}{\mathrel{\in\in}}

\newcommand*{\Rordam}{R{\o}rdam}

\newcommand{\Reps}{\mathfrak{Rep}}
\newcommand{\CoSheaf}{\mathfrak{CoSh}}
\newcommand{\PreCoSheaf}{\mathfrak{PreCoSh}}

\theoremstyle{plain}

\numberwithin{equation}{section}

\theoremstyle{plain}
\newtheorem{theorem}[equation]{Theorem}
\newtheorem{lemma}[equation]{Lemma}

\newtheorem{corollary}[equation]{Corollary}

\newtheorem{proposition}[equation]{Proposition}

\theoremstyle{definition}
\newtheorem{definition}[equation]{Definition}

\theoremstyle{remark}
\newtheorem{remark}[equation]{Remark}
\newtheorem{remarks}[equation]{Remarks}

\title[Kirchberg X-algebras with real rank zero...]{Kirchberg X-algebras with real rank zero\\ and intermediate cancellation}

\author{Rasmus Bentmann}

\address{Department of Mathematical Sciences\\University of
Copenhagen\\Universitets\-parken~5\\2100 Copenhagen \O \\Denmark}
\email{bentmann@math.ku.dk}

\thanks{The author was supported by the Danish National Research Foundation (DNRF) through the Centre for Symmetry and Deformation and by the Marie Curie Research Training Network EU-NCG}

\begin{document}

\begin{abstract}
A universal coefficient theorem is proved for \Cstar{}al\-ge\-bras over an arbitrary finite $T_0$\nb-space~$X$ which have vanishing boundary maps. Under bootstrap assumptions, this leads to a complete classification of unital/stable real-rank-zero Kirchberg $X$\nb-al\-ge\-bras with intermediate cancellation. Range results are obtained for (unital) purely infinite graph \Cstar{}algebras with intermediate cancellation and Cuntz--Krieger algebras with intermediate cancellation. Permanence results for extensions of these classes follow.
\end{abstract}

\maketitle

\section{Introduction}

Since Eberhard Kirchberg's groundbreaking classification theorem for non-simple $\Cuntz_\infty$\nb-ab\-sor\-bing nuclear \Cstar{}al\-ge\-bras~\cite{Kirchberg}, much effort has gone into the task of deciding when two separable \Cstar{}al\-ge\-bras over a topological space~$X$ are $\KK(X)$\nb-equiv\-a\-lent. This is a hard task even when~$X$ is a finite space. The usual way to go is to prove equivariant versions of the \emph{universal coefficient theorem} of Rosenberg and Schochet~\cite{RS}. For \emph{some} spaces, such have been established in \cites{Bonkat:Thesis,Restorff:Thesis,MN:Filtrated,Bentmann:Thesis,BK}. In~\cite{Bentmann-Meyer}, a complete classification in purely algebraic terms of objects in the equivariant bootstrap class $\Bootstrap(X)\subset\KKcat(X)$ up to $\KK(X)$\nb-equiv\-a\-lence is given under the assumption that~$X$ is a so-called \emph{unique path space.} Nevertheless, it seems fair to state that, for \emph{most} finite spaces, no classification is available at the present time.

In this note we establish a universal coefficient theorem computing the groups $\KK_*(X;A,B)$ which holds for all finite $T_0$\nb-spa\-ces~$X$---but only under certain $\K$\nb-the\-o\-ret\-i\-cal assumptions on~$A$. More precisely, we have to ask that the boundary maps in all six-term exact sequences arising from inclusions of distinguished ideals vanish. If~$A$ is separable, purely infinite and tight over~$X$, this condition is equivalent to $A$~having real rank zero and the following non-stable $\K$\nb-the\-ory property suggested to us by Mikael \Rordam: if~$p$ and~$q$ are projections in~$A$ which generate the same ideal and which give rise to the same element in $\K_0(A)$, then~$p$ and~$q$ are Murray-von Neumann equivalent. This property has been considered earlier by Lawrence G.~Brown~\cite{Brown:intermediate_cancellation}. Since the property is stronger than Brown-Pedersen's weak cancellation property and weaker than Rieffel's strong cancellation property (compare~\cite{Brown-Pedersen:Non-stable}), it is referred to as \emph{intermediate cancellation}.

The invariant appearing in our universal coefficient theorem, denoted by~$\XK$, is relatively simple: for a point $x\in X$, let $U_x$ denote its minimal open neighbourhood. Then $\XK(A)$ consists of the collection $\{\K_*\bigbr{A(U_x)}\mid x\in X\}$ together with the natural maps induced by the ideal inclusions $A(U_x)\hookrightarrow A(U_y)$ for $U_x\subseteq U_y$. Hence~$\XK(A)$ can be regarded as a representation of the partially ordered set~$X$ with values in countable $\Z/2$-graded Abelian groups. Equivalently, we may view~$\XK(A)$ as a countable $\Z/2$-graded module over the integral incidence algebra~$\Z X$ of~$X$. The fact that the ring~$\Z X$ itself is ungraded allows us to show that the universal coefficient sequence for $\KK_*(X;A,B)$ splits if both~$A$ and~$B$ have vanishing boundary maps and that an object in the equivariant bootstrap class~$\Bootstrap(X)$ with vanishing boundary maps is $\KK(X)$\nb-equiv\-a\-lent to a commutative \Cstar{}algebra over~$X$.

A Kirchberg $X$\nb-algebra is a nuclear purely infinite separable tight \Cstar{}algebra over~$X$. Combining our universal coefficient theorem with Kirchberg's theorem, we find that the invariant~$\XK$ strongly classifies stable real-rank-zero Kirchberg $X$\nb-al\-ge\-bras with intermediate cancellation and simple subquotients in the bootstrap class up to \Star{}isomorphism over~$X$.

We also describe the range of the invariant~$\XK$ on this class of \Cstar{}algebras over~$X$, but only in the case that~$X$ is a unique path space. To this aim, we use a second invariant denoted by~$\OK$. It is defined similarly to~$\XK$ but it contains the $\K$-groups of \emph{all} distinguished ideals. The target category of~$\OK$ is the category of precosheaves on the topology of~$X$ with values in countable $\Z/2$-graded Abelian groups. It turns out that the range of~$\OK$ on the class of stable real-rank-zero Kirchberg $X$\nb-al\-ge\-bras with intermediate cancellation and simple subquotients in the bootstrap class consists precisely of those precosheaves which satisfy a certain cosheaf condition and have injective structure maps; following Bredon~\cite{Bredon:Cosheaves}, we call these \emph{flabby cosheaves}.

Appealing to the so-called meta theorem \cite{Eilers-Restorff-Ruiz:generalized_meta_theorem}*{Theorem~3.3}, we can achieve strong classification also in the unital case. The invariant in this case, denoted by~$\OK^+$, consists of the functor~$\OK$ together with the unit class in the $\K_0$\nb-group of the whole \Cstar{}algebra.

We apply our results to the classification programme of (purely infinite) graph \Cstar{}al\-ge\-bras. Here real rank zero comes for free, as do separability, nuclearity and bootstrap assumptions. We determine the range of the invariant~$\OK$ on the class of purely infinite tight graph \Cstar{}al\-ge\-bras over~$X$ with intermediate cancellation. We also determine the range of the invariant~$\OK^+$ on the class of unital purely infinite tight graph \Cstar{}al\-ge\-bras over~$X$ with intermediate cancellation and on the class of tight Cuntz--Krieger algebras over~$X$ with intermediate cancellation. Here we use a result from~\cite{Arklint-Bentmann-Katsura:Reduction} that allows to construct graph \Cstar{}al\-ge\-bras with prescribed $\K$-theory data.

As an application, we show that the class of Cuntz--Krieger algebras with intermediate cancellation is, in a suitable sense, stable under extensions (see Theorem~\ref{thm:extensions} for the precise statement). A similar result is obtained in~\cite{Arklint-Bentmann-Katsura:Reduction}*{Corollary 9.15}, but under different assumptions: in~\cite{Arklint-Bentmann-Katsura:Reduction} we make assumptions on the primitive ideal space to make the classification machinery work; in \emph{this} article we use intermediate cancellation to achieve that. Similar permanence results hold for (unital) purely infinite graph \Cstar{}algebras with intermediate cancellation.

\section{Preliminaries}

Throughout, let~$X$ be an arbitrary finite $T_0$\nb-space. A subset of~$X$ is \emph{locally closed} if it is the difference of two open subsets of~$X$. Every point $x\in X$ possesses a smallest open neighbourhood denoted by~$U_x$. The \emph{specialization preorder} on~$X$ is the partial order defined such that $x\leq y$ if and only if $U_y\subseteq U_x$. For two points $x,y\in X$, there is an arrow from~$y$ to~$x$ in the Hasse diagram associated to the specialization preorder on~$X$ if and only if~$y$ is a closed point in $U_x\setminus\{x\}$; in this case we write $y\to x$. We say that~$X$ is a \emph{unique path space} if every pair of points in~$X$ is connected by at most one directed path in the Hasse diagram associated to the specialization preorder on~$X$.

A \emph{\Cstar{}algebra over~$X$} is a pair $(A,\psi)$ consisting of a \Cstar{}algebra~$A$ and a continuous map $\psi\colon\Prim(A)\to X$. The pair $(A,\psi)$ is called \emph{tight} if the map~$\psi$ is a homeomorphism. We usually omit the map~$\psi$ in order to simplify notation. There is a lattice isomorphism between the open subsets in $\Prim(A)$ and the ideals in~$A$. Hence every open subset~$U$ of~$X$ gives rise to a \emph{distinguished ideal}~$A(U)$ in~$A$. A~\emph{\Star{}ho\-mo\-mor\-phism over~$X$} is a \Star{}ho\-mo\-mor\-phism mapping distinguished ideals into corresponding distinguished ideals. We obtain the category $\Cstarcat(X)$ of \Cstar{}al\-ge\-bras over~$X$ and \Star{}ho\-mo\-mor\-phisms over~$X$. Any locally closed subset~$Y$ of~$X$ determines a \emph{distinguished subquotient}~$A(Y)$ of~$A$. There is a natural way to regard the subquotient~$A(Y)$ as a \Cstar{}algebra over~$Y$. For a point $x\in X$, we let $i_x\C$ denote the \Cstar{}algebra over~$X$ given by the \Cstar{}algebra of complex numbers~$\C$ with the map $\Prim(\C)\to X$ taking the unique primitive ideal in~$\C$ to~$x$. For more details on \Cstar{}algebras over topological spaces, see~\cite{MN:Bootstrap}.

Eberhard Kirchberg developed a version of Kasparov's $\KK$\nb-theory for separable \Cstar{}algebras over~$X$ in~\cite{Kirchberg} denoted by $\KK(X)$. In~\cite{MN:Bootstrap}, Ralf Meyer and Ryszard Nest establish basic properties of the resulting category~$\KKcat(X)$, describe a natural triangulated category structure on it, and give an appropriate definition of the equivariant bootstrap class $\Bootstrap(X)\subset\KKcat(X)$: it is the smallest triangulated subcategory of $\KKcat(X)$ that contains the object set $\{i_x\C\mid x\in X\}$ and is closed under countable direct sums. The usual bootstrap class in~$\KKcat$ of Rosenberg and Schochet is denoted by~$\Bootstrap$. The translation functor on $\KKcat(X)$ is given by suspension and denoted by~$\Sigma$. The category~$\KKcat(X)$ is tensored over~$\KKcat$; in particular, we can talk about the stabilization $A\otimes\Compacts$ of an object~$A$ in~$\KKcat(X)$. Here~$\Compacts$ denotes the \Cstar{}algebra of compact operators on some countably infinite-dimensional Hilbert space.

For an object~$M$ in a $\Z/2$-graded category, we write $M_0$~for the even part, $M_1$~for the odd part and~$M[1]$ for the shifted object. If~$N$ is an object in the ungraded category, we let~$N[i]$ denote the corresponding graded object concentrated in degree~$i$. We write $C\inOb\mathcal C$ to denote that~$C$ is an object in a category~$\mathcal C$.

\section{Vanishing boundary maps}

In this section, we introduce two $\K$-theoretical conditions for \Cstar{}algebras over~$X$ that are sufficient, as we shall see later, to obtain a universal coefficient theorem. We provide alternative formulations of these conditions for separable purely infinite tight \Cstar{}algebras over~$X$.

Given a \Cstar{}algebra~$A$ over~$X$ and open subsets $U\subseteq V\subseteq X$, we have a six-term exact sequence
\begin{equation}
	\label{eq:six-term_sequence}
  \begin{gathered}
  \begin{tikzpicture}
    \matrix(m)[cd]{
    \K_1\bigbr{A(U)} & \K_1\bigbr{A(V)} & \K_1\bigbr{A(V)/A(U)} \\
    \K_0\bigbr{A(V)/A(U)} & \K_0\bigbr{A(V)} & \K_0\bigbr{A(U)}.\\
    };
    \draw[cdar] (m-1-1) -- (m-1-2);
    \draw[cdar] (m-2-2) -- (m-2-1);
    \draw[cdar] (m-1-2) -- (m-1-3);
    \draw[cdar] (m-2-3) -- (m-2-2);
    \draw[cdar] (m-1-3) -- node {\(\partial_1\)} (m-2-3);
    \draw[cdar] (m-2-1) -- node {\(\partial_0\)} (m-1-1);
  \end{tikzpicture}
  \end{gathered}
\end{equation}

\begin{definition}
Let~$A$ be a \Cstar{}algebra over~$X$. We say that \emph{$A$ has vanishing index maps} if the map $\partial_1\colon\K_1\bigbr{A(V)/A(U)}\to\K_0\bigbr{A(U)}$ vanishes for all open subsets $U\subseteq V\subseteq X$. Similarly, we  say that \emph{$A$ has vanishing exponential maps} if the map $\partial_0\colon\K_0\bigbr{A(V)/A(U)}\to\K_1\bigbr{A(U)}$ vanishes for all open subsets $U\subseteq V\subseteq X$. We say that \emph{$A$ has vanishing boundary maps} if it has vanishing index maps and vanishing exponential maps.
\end{definition}

\begin{remarks}
If~$A$ is a tight \Cstar{}algebra over~$X$ then~$A$ has vanishing exponential maps if and only if the underlying \Cstar{}algebra of~$A$ is $\K_0$\nb-liftable in the sense of \cite{Rordam_Pasnicu:PI}*{Definition~3.1}.

In the definition above, we could replace the subset $V\subseteq X$ with~$X$, but to us the definition seems more natural as it stands.

Another, a priori \emph{stronger} condition consists in the vanishing of all boundary maps arising from inclusions of distinguished \emph{subquotients}. The following lemma shows that this assumption is in fact equivalent to the one in our definition.
\end{remarks}

\begin{lemma}
  \label{lem:all_boundary_maps}
Let $Y\subseteq X$ be locally closed. Let $U\subseteq Y$ be relatively open. Write $C=Y\setminus U$. Let~$A$ be a \Cstar{}algebra over~$X$ with vanishing index\textup{/}exponential maps. Then the index\textup{/}exponential map corresponding to the extension $A(U)\rightarrowtail A(Y)\twoheadrightarrow A(C)$ vanishes, too.
\end{lemma}

\begin{proof}
Write $Y=V\setminus W$ as the difference of two open subsets $W\subseteq V\subseteq X$. Consider the morphism of extensions of distinguished subquotients
  \[
  \begin{tikzpicture}
    \matrix(m)[cd]{
      A(V\setminus C)&A(V)&A(C)\\
      A(U)&A(Y)&A(C).\\
    };
    \path[>->] (m-1-1) edge (m-1-2);
    \path[>->] (m-2-1) edge (m-2-2);
    \path[->>] (m-1-2) edge (m-1-3);
    \path[->>] (m-2-2) edge (m-2-3);
    \draw[double distance = 1.5pt] (m-1-3) -- (m-2-3);
    \path[->>] (m-1-2) edge (m-2-2);
    \path[->>] (m-1-1) edge (m-2-1);
  \end{tikzpicture}
  \]
The first extension has vanishing index\textup{/}exponential map by assumption. By naturality, the same follows for the second extension.
\end{proof}

\begin{proposition}
  \label{pro:vanishing_boundaries_and extensions}
Let~$U\subseteq X$ be an open subset and write $C=X\setminus U$. Let~$A$ be a \Cstar{}algebra over~$X$. Then~$A$ has vanishing index maps if and only if the following hold:
\begin{itemize}
\item $A(U)\inOb\Cstarcat(U)$ has vanishing index maps,
\item $A(C)\inOb\Cstarcat(C)$ has vanishing index maps,
\item the index map $\K_1\bigbr{A(C)}\to\K_0\bigbr{A(U)}$ vanishes.
\end{itemize}
An analogous statement holds for vanishing exponential maps.
\end{proposition}

\begin{proof}
We will only prove the statement for index maps, the case of exponential maps being entirely analogous. By the previous lemma, the three conditions are necessary. To show that they are also sufficient, we consider an open subset $V\subseteq X$. It suffices to check that the map $\K_0\bigbr{A(V)}\to\K_0\bigbr{A(X)}$ is injective. We consider the morphism of extensions of distinguished subquotients
  \[
  \begin{tikzpicture}
    \matrix(m)[cd]{
      A(U\cap V)&A(U)&A\bigbr{U\setminus (U\cap V)}\\
      A(V)&A(U\cup V)&A\bigbr{(U\setminus (U\cap V)}.\\
    };
    \path[>->] (m-1-1) edge (m-1-2);
    \path[>->] (m-2-1) edge (m-2-2);
    \path[->>] (m-1-2) edge (m-1-3);
    \path[->>] (m-2-2) edge (m-2-3);
    \draw[double distance = 1.5pt] (m-1-3) -- (m-2-3);
    \path[>->] (m-1-2) edge (m-2-2);
    \path[>->] (m-1-1) edge (m-2-1);
  \end{tikzpicture}
  \]
By the first condition, the upper extension has vanishing index map. By naturality, so has the second. Hence the map $\K_0\bigbr{A(V)}\to\K_0\bigbr{A(U\cup V)}$ is injective. By the second and third condition, the composition
\[
\K_1\bigbr{A(X)}\to\K_1\bigbr{A(C)}\to\K_1\bigbr{A(X\setminus (U\cup V))}
\]
is surjective. By the six-term exact sequence, the map $\K_0\bigbr{A(U\cup V)}\to\K_0\bigbr{A(X)}$ is thus injective. The result follows.
\end{proof}

\begin{corollary}
  \label{cor:vanishing_boundary_maps_criterion}
Let~$A$ be a \Cstar{}algebra over~$X$. Then~$A$ has vanishing index\textup{/}ex\-po\-nen\-tial maps if and only if the index\textup{/}exponential map of the extension
\[
 A(U_x\setminus\{x\})\rightarrowtail A(U_x)\twoheadrightarrow A(\{x\})
\]
vanishes for every point $x\in X$.
\end{corollary}

\begin{proof}
Again, we will only prove the statement for index maps. The condition is clearly necessary. In order to prove sufficiency, we choose a filtration
\[
\emptyset= V_0\subsetneq V_1\subsetneq\dotsb\subsetneq V_\ell=X,
\]
of~$X$ by open subsets~$V_j$ such that $V_j\setminus V_{j-1}=\{x_j\}$ is a singleton for all \(j=1,\dotsc,\ell\). By naturality of the index map, the condition implies that the index map of the extension
\[
 A(V_{j-1})\rightarrowtail A(V_j)\twoheadrightarrow A(\{x_j\})
\]
vanishes for all \(j=1,\dotsc,\ell\). A repeated application of Proposition~\ref{pro:vanishing_boundaries_and extensions} gives the desired result, because a \Cstar{}al\-ge\-bra over the one-point space automatically has vanishing index maps.
\end{proof}

Now we turn to the description of separable purely infinite tight \Cstar{}al\-ge\-bras over~$X$ with vanishing boundary maps.

\begin{proposition}
  \label{pro:characterization_of_KXA_with_vanishing_exponential_maps}
A separable purely infinite tight \Cstar{}algebra over~$X$ has vanishing exponential maps if and only if its underlying \Cstar{}algebra has real rank zero.
\end{proposition}

\begin{proof}
This is a special case of \cite{Rordam_Pasnicu:PI}*{Theorem~4.2} because~$X$ is a quasi-com\-pact space; see also \cite{Rordam_Pasnicu:PI}*{Example~4.8}.
\end{proof}

The following definition has been suggested to us by Mikael \Rordam; it has been considered earlier by Lawrence G.\ Brown~\cite{Brown:intermediate_cancellation}.

\begin{definition}
	\label{def:intermediate_cancellation}
A \Cstar{}algebra~$A$ has \emph{intermediate cancellation} if the following holds: if~$p$ and~$q$ are projections in~$A$ which generate the same ideal and which give rise to the same element in $\K_0(A)$, then $p\sim q$ (that is, the projections~$p$ and~$q$ are Murray-von Neumann equivalent).
\end{definition}

\begin{lemma}
  \label{lem:K0_of_purely_infinite}
Let~$A$ be a separable purely infinite \Cstar{}algebra with finite ideal lattice. Then
\[
 \K_0(A)=\{[p]\mid\textup{$p$ is a full projection in~$A$}\}.
\]
Moreover, if~$p$ and~$q$ are full projections in~$A$ with $[p]=[q]$ in $\K_0(A)$, then $p\sim q$.
\end{lemma}

\begin{proof}
It follows from \cite{Kirchberg_Rordam:Non-simple}*{Theorem 4.16}, that every non-zero projection in~$A$ is properly infinite.  The lemma thus follows from \cite{Rordam:Classification_of_nuclear}*{Proposition~4.1.4} because~$A$ contains a full projection by \cite{Rordam_Pasnicu:PI}*{Proposition~2.7}.
\end{proof}

\begin{proposition}
  \label{pro:characterization_of_KXA_with_vanishing_index_maps}
A separable purely infinite tight \Cstar{}algebra over~$X$ has vanishing index maps if and only if its underlying \Cstar{}algebra has intermediate cancellation.
\end{proposition}

\begin{proof}
By \cite{Kirchberg_Rordam:Non-simple}*{Proposition~4.3}, every ideal in~$A$ is purely infinite. The proposition follows from applying Lemma~\ref{lem:K0_of_purely_infinite} to every ideal of~$A$.
\end{proof}

\begin{corollary}
	\label{cor:intermediate_cancellation_and_extensions}
Let $I\rightarrowtail A\twoheadrightarrow B$ be an extension of \Cstar{}algebras. Assume that~$A$ is  separable, purely infinite and has finite ideal lattice. Then~$A$ has intermediate cancellation if and only if the following hold:
\begin{itemize}
\item $I$ has intermediate cancellation,
\item $B$ has intermediate cancellation,
\item the index map $\K_1(B)\to\K_0(I)$ vanishes.
\end{itemize}
\end{corollary}

\begin{proof}
Combine Propositions \ref{pro:vanishing_boundaries_and extensions} and \ref{pro:characterization_of_KXA_with_vanishing_index_maps}.
\end{proof}

\begin{remark}
  \label{rem:rrzero}
The analogue of Corollary \ref{cor:intermediate_cancellation_and_extensions} with real rank zero replacing intermediate cancellation and the exponential map $\K_0(B)\to\K_1(I)$ replacing the index map $\K_1(B)\to\K_0(I)$ is well-known and holds in much greater generality; see~\cites{Brown-Pedersen:RR0,Lin-Rordam:Extensions_of_inductive_limits}.
\end{remark}

\section{Representations and cosheaves}
  \label{sec:reps_and_cosheaves}
  
In this section, we introduce two $\K$-theoretical invariants for \Cstar{}algebras over~$X$ that are well-adapted to algebras with vanishing boundary maps. First we define their target categories.

We associate the following two partially ordered sets to~$X$:
\begin{itemize}
 \item the set~$X$ itself, equipped with the specialization preorder;
 \item the collection $\Open(X)$ of open subsets of~$X$, partially ordered by inclusion.
\end{itemize}
The map $X^\op\to\Open(X),\; x\mapsto U_x$ is an embedding of partially ordered sets. Here~$X^\op$ denotes the set~$X$ with reversed partial ordering. For the following definition, recall that every partially ordered set can be viewed as a category such that $\Hom(x,y)$ has one element, denoted by~$i_x^y$, if $x\leq y$ and zero elements otherwise.

\begin{definition}
Let $\GCAb$ be the category of countable $\Z/2$-graded Abelian groups. A \emph{representation} of~$X$ is a covariant functor $X^\op\to\GCAb$. A \emph{precosheaf} on $\Open(X)$ is a covariant functor $\Open(X)\to\GCAb$. A precosheaf $M\colon\Open(X)\to\GCAb$ is a \emph{cosheaf} if, for every $U\in\Open(X)$ and every open covering $\{U_j\}_{j\in J}$ of~$U$, the sequence
\begin{multline*}
  \label{eq:cosheaf_condition_on_topology}
\bigoplus_{j,k\in J} M(U_j\cap U_k)
\xrightarrow{\left(M(i_{U_j\cap U_k}^{U_j})-M(i_{U_j\cap U_k}^{U_k})\right)}
\bigoplus_{j\in J} M(U_j)
\xrightarrow{\left(M(i_{U_j}^U)\right)}
M(U) \longrightarrow 0
\end{multline*}
is exact. Letting morphisms be natural transformations of functors, we define the category $\Reps(X)$ of representations of~$X$, the category $\PreCoSheaf\bigbr{\Open(X)}$ of precosheaves over~$\Open(X)$ and the category $\CoSheaf\bigbr{\Open(X)}$ of cosheaves over~$\Open(X)$.
\end{definition}

The notion of cosheaf was introduced by Bredon~\cite{Bredon:Cosheaves}. Just like sheaves, cosheaves are determined by their behaviour on a basis. This is made precise in the following definition and lemma.

\begin{definition}
Let $\Res\colon\CoSheaf\bigbr{\Open(X)}\to\Reps(X)$ be the \emph{restriction} functor given by
\[
\Res(M)(x)=M(U_x),\quad\Res(M)(i_x^y)=M\left(i_{U_x}^{U_y}\right).
\]
Let $\Colim\colon\Reps(X)\to\CoSheaf\bigbr{\Open(X)}$ be the functor that extends a representation~$M$ of~$X$ to a cosheaf on~$\Open(X)$ in a way such that $\bigbr{\Colim(M)}(U)$ is given by the cokernel of the map
\begin{equation*}
  \label{eq:def_of_colim}
\bigoplus_{x,y\in U}\bigoplus_{z\in U_x\cap U_y} M(z)
\xrightarrow{\bigbr{M(i_z^x)-M(i_z^y)}}
\bigoplus_{x\in U} M(x)
\end{equation*}and $\Colim(M)(i_U^V)$ is induced by the obvious inclusions $\bigoplus_{x\in U} M(x)\subseteq\bigoplus_{x\in V} M(x)$ and $\bigoplus_{x,y\in U}\bigoplus_{z\in U_x\cap U_y} M(z)\subseteq\bigoplus_{x,y\in V}\bigoplus_{z\in U_x\cap U_y} M(z)$. We call $\Colim(M)$ the \emph{associated cosheaf} of the representation~$M$.
\end{definition}

\begin{lemma}
	\label{lem:res_and_colim_inverse}
The functor $\Colim$ indeed takes values in cosheaves on~$\Open(X)$.
The functors $\Res$ and $\Colim$ are mutually inverse equivalences of categories.
\end{lemma}

\begin{proof}
The corresponding statements for sheaves are well-known: see, for instance, \cite{stacks-project}*{Lemmas~\href{http://stacks.math.columbia.edu/tag/009N}{009N} and~\href{http://stacks.math.columbia.edu/tag/009O}{009O}}. Our dual version for cosheaves is a straight-forward analogue. Notice that $\{U_x\mid x\in X\}$ is a basis for the topology on~$X$ with the special property that every covering of an open set in it must contain this open set. Hence every precosheaf on this basis is already a cosheaf.
\end{proof}

\begin{definition}
The \emph{integral incidence algebra} $\Z X$ of~$X$ is the free Abelian group generated by elements \(i_x^y\) for all pairs \((x,y)\) with \(y\leq x\) equipped with the unique bilinear multiplication such that \(i_z^w i_x^y\) equals \(i_x^w\) if $y=z$ and otherwise is zero. By \(\GCMod{\Z X}\), we denote the category of countable $\Z/2$-graded left-modules over~$\Z X$.

The categories $\Reps(X)$ and \(\GCMod{\Z X}\) are canonically equivalent; we will identify them tacitly. For every point $x\in X$, we have a projective module $P^x\defeq\Z X\cdot i_x^x$ in \(\GCMod{\Z X}\) associated to the idempotent element~$i_x^x$. Its entries are given by
\[
(P^x)(y)=
\begin{cases}
\Z[0]\cdot i_x^y & \textup{if $y\leq x$} \\
0 & \textup{otherwise}
\end{cases}
\]
and the map $(P^x)(i_y^z)$ for $y\geq z$ is an isomorphism if $x\geq y$ and zero otherwise.
\end{definition}

\begin{definition}[\cite{Bredon:Cosheaves}*{\S 1}]
A cosheaf on~$\Open(X)$ is called \emph{flabby} if all its structure maps are injective.
\end{definition}

The following is our key-lemma towards the universal coefficient theorem.

\begin{lemma}
	\label{lem:proj_dim_1}
Let~$M$ be a representation of~$X$ such that the associated cosheaf $\Colim(M)$ on~$\Open(X)$ is flabby. Then~$M$ has a projective resolution of length~$1$.
\end{lemma}

\begin{proof}
As before, we may choose a filtration
\[
\emptyset= V_0\subsetneq V_1\subsetneq\dotsb\subsetneq V_\ell=X,
\]
of~$X$ by open subsets~$V_j$ such that $V_j\setminus V_{j-1}=\{x_j\}$ is a singleton for all \(j=1,\dotsc,\ell\). For $V\in\Open(X)$ we define a representation $P_V M$ of~$X$ by
\[
 (P_V M)(x) = \Colim(M)(V\cap U_x).
\]
Since $\Colim(M)$ is flabby, we obtain a filtration
\[
0= P_{V_0} M\subseteq P_{V_1} M\subseteq\dotsb\subseteq P_{V_\ell} M=M.
\]
It follows from the so-called Horseshoe Lemma that an extension of modules with projective resolutions of length~$1$ also has a projective resolution of length~$1$. Hence it remains to show that the subquotients $Q^j\defeq P_{V_j} M/P_{V_{j-1}} M$ in our filtration have resolutions of length~$1$.

Let us describe the modules~$Q^j$ explicitly. If $x_j\not\in U_x$, then we have
\[
(P_{V_j} M)(x)=\Colim(M)(V_j\cap U_x)=\Colim(M)(V_{j-1}\cap U_x)=(P_{V_{j-1}} M)(x),
\]
so that $Q^j(x)=0$. Now we assume $x_j\in U_x$. We fix~$y\in X$ with $x\in U_y$ and abbreviate $C\defeq\Colim(M)$. Since~$C$ is a cosheaf, we have a pushout diagram
\[
\xymatrix{
 C(V_{j-1}\cap U_x)\ar[r]\ar[d] & C(V_{j}\cap U_x)\ar[d] \\
 C(V_{j-1}\cap U_y)\ar[r] & C(V_{j}\cap U_y).
}
\]
Since pushouts preserve cokernels, we obtain that the map $Q^j(x)\to Q^j(y)$ is an isomorphism. In conclusion, we may identify $Q^j\cong P^{x_j}\otimes G^j$, where~$G^j$ is some countable $\Z/2$-graded Abelian group. A projective resolution of length~$1$ for~$Q^j$ can thus be obtained by tensoring the projective module~$P^{x_j}$ with a resolution of~$G^j$.
\end{proof}

Now we turn to the definition of our $\K$-theoretical invariants.

\begin{definition}
We define a functor $\XK\colon\KKcat(X)\to\Reps(X)\cong\GCMod{\Z X}$ as follows: set
\[
\XK(A)(x)=\K_*\bigbr{A(U_x)}
\]
and let $\XK(A)(i_x^y)$ be the map induced by the ideal inclusion $A(U_x)\hookrightarrow A(U_y)$.

Similarly, we define $\OK\colon\KKcat(X)\to\PreCoSheaf\bigbr{\Open(X)}$ by $\OK(A)(U)=\K_*\bigbr{A(U)}$ and let the structure maps be the homomorphisms induced by the ideal inclusions.
\end{definition}

We have an identity of functors $\Res\circ\OK=\XK$.

\begin{lemma}
  \label{lem:flabby_cosheaf}
A \Cstar{}algebra~$A$ over~$X$ has vanishing boundary maps if and only if $\OK(A)$ is a flabby cosheaf.
\end{lemma}

\begin{proof}
Suppose that~$A$ has vanishing boundary maps. By an inductive argument as in \cite{Bredon:Cosheaves}*{Proposition~1.3}, it suffices to verify the cosheaf condition for all coverings consisting of two open sets. This case reduces to the Mayer-Vietoris sequence. The six-term exact sequence~\eqref{eq:six-term_sequence} shows that $\OK(A)$ is flabby.

Conversely, if $\OK(A)$ is a flabby cosheaf, the six-term exact sequence~\eqref{eq:six-term_sequence} shows that~$A$ has vanishing boundary maps.
\end{proof}

It follows from Lemma~\ref{lem:res_and_colim_inverse} that, on the full subcategory of \Cstar{}algebras over~$X$ with vanishing boundary maps, we have a natural isomorphism $\Colim\circ\XK\cong\OK$.

\begin{remark}
Instead of working with $\K$-theory groups of distinguished ideals, we could define similar invariants in terms of $\K$-theory groups of distinguished quotients. This would not make a difference for the universal coefficient theorem in the next section. However, our choice of definition interacts more nicely with the invariant~$\FK_\Catgunnar$ that we will use in~\S\ref{sec:range_on_graph_algebras}.
\end{remark}

For reference in future work, we record the following lemma.

\begin{lemma}
Let $A$ be a \Cstar{}algebra over~$X$ with vanishing boundary maps such that the Abelian group $\K_*\bigl(A(Y)\bigr)$ is free for every locally closed subset $Y\subseteq X$. Then $\XK(A)$ is projective.
\end{lemma}

\begin{proof}
By Lemma \ref{lem:flabby_cosheaf}, the cosheaf $\OK(A)$ is flabby. We follow the proof of Lemma \ref{lem:proj_dim_1}. Our freeness assumption implies that the Abelian groups $G$ coming up in the proof are free: the six-term exact sequence shows that
\[
G^j = \K_*\bigl(A(U_x)\bigr) / \K_*\bigl(A(U_x\setminus\{x\})\bigr) \cong\K_*\bigl(A(\{x\})\bigr).
\]
Hence $\XK(A)$ is an iterated extension of projective modules and thus itself projective.
\end{proof}

\section{A universal coefficient theorem}

In this section, we establish a universal coefficient theorem for \Cstar{}algebras over~$X$ with vanishing boundary maps. We discuss the splitting of the resulting short exact sequence and the realization of objects in the bootstrap class as commutative algebras.

We describe how the invariant~$\XK$ fits into the framework for homological algebra in triangulated categories developed by Meyer and Nest in~\cite{Meyer-Nest:Homology_in_KK}. The set-up is given by the triangulated category $\KKcat(X)$ and the stable homological ideal $\Ideal\defeq\ker(\XK)$, the kernel of~$\XK$ on morphisms. Using the adjointness relation
\begin{equation}
  \label{eq:adjointness_relation}
 \KK_*(X;i_x\C,A)\cong\KK_*\bigbr{\C,A(U_x)}\cong\K_*\bigbr{A(U_x)}
\end{equation}
from \cite{MN:Bootstrap}*{Proposition~3.13} and machinery from~\cite{Meyer-Nest:Homology_in_KK}, one can easily show the following (a slightly more detailed account for the particular example at hand is given in~\cite{Bentmann-Meyer}*{\S4}):
\begin{itemize}
\item
the ideal~\(\Ideal\) has enough projective objects,
\item 
the functor~$\XK$ is the \emph{universal} $\Ideal$-exact stable homological functor,
\item
\(A\) belongs to~\(\Bootstrap(X)\) if and only if \(\KK_*(X;A,B)=0\) for all \(\Ideal\)\nb-contractible~\(B\).
\end{itemize}
These facts allow us to apply the abstract universal coefficient theorem \cite{Meyer-Nest:Homology_in_KK}*{Theorem~66} to our concrete setting. We abbreviate \(\AbelianCat\defeq\GCMod{\Z X}\).

\begin{theorem}
Let~$A$ and~$B$ be separable \Cstar{}algebras over~$X$. Assume that~$A$ belongs to $\Bootstrap(X)$ and has vanishing boundary maps. Then there is a natural short exact sequence of $\Z/2$-graded Abelian groups
\begin{equation}
  \label{eq:uct}
\Ext^1_\AbelianCat\bigbr{\XK(A)[1],\XK(B)}
\rightarrowtail\KK_*(X;A,B)
\twoheadrightarrow\Hom_\AbelianCat\bigbr{\XK(A),\XK(B)}.
\end{equation}
\end{theorem}

\begin{proof}
By~\cite{Meyer-Nest:Homology_in_KK}*{Theorem~66}, we only have to check that $\XK(A)$ has a projective resolution of length~$1$. This follows from Lemmas~\ref{lem:flabby_cosheaf} and~\ref{lem:proj_dim_1}.
\end{proof}

\begin{corollary}
  \label{cor:lifting}
Let~$A$ and~$B$ be separable \Cstar{}algebras over~$X$. Assume that~$A$ and~$B$ belong to $\Bootstrap(X)$ and that~$A$ has vanishing boundary maps. Then every isomorphism $\XK(A)\cong\XK(B)$ in~$\AbelianCat$ can be lifted to a $\KK(X)$-equivalence.
\end{corollary}

\begin{proof}
Since~$A$ has vanishing boundary maps, the module $\XK(A)\cong\XK(B)$ has a projective resolution of length~$1$ by Lemmas~\ref{lem:flabby_cosheaf} and~\ref{lem:proj_dim_1}. Hence the result follows from the universal coefficient theorem \cite{Meyer-Nest:Homology_in_KK}*{Theorem~66} by a standard argument; see, for instance, \cite{Blackadar:Kbook}*{Proposition~23.10.1} or \cite{MN:Filtrated}*{Corollary~4.6}.
\end{proof}

\begin{proposition}
  \label{pro:splitting}
Let~$A$ and~$B$ be separable \Cstar{}algebras over~$X$. Assume that~$A$ belongs to $\Bootstrap(X)$ and that~$A$ and~$B$ have vanishing boundary maps. Then the short exact sequence~\eqref{eq:uct} splits \textup{(}unnaturally\textup{)}.
\end{proposition}

\begin{proof}
For this result, it is crucial that the ring~$\Z X$ itself is ungraded. We can thus imitate the proof from \cite{Blackadar:Kbook}*{\S23.11}: we have direct sum decompositions $\XK(A)\cong M_0\oplus M_1[1]$ and $\XK(B)\cong N_0\oplus N_1[1]$ where~$M_i$ and~$N_i$ are ungraded $\Z X$\nb-modules of projective dimension at most~$1$. By a simple argument based on the universality of the functor~$\XK$ (compare \cite{MN:Filtrated}*{Theorem~4.8}), we can find objects~$A_i$ and~$B_i$ in~$\Bootstrap(X)$ such that $\XK(A_i)\cong M_i[0]$ and $\XK(B_i)\cong N_i[0]$ for $i\in\{0,1\}$. By Corollary~\ref{cor:lifting}, there is a (non-canonical) $\KK(X)$-equivalence $A\cong A_1\oplus\Sigma A_2$. Using the universal coefficient theorem, we can find an element $f\in\KK_0(X;B, B_1\oplus\Sigma B_2)$ inducing an isomorphism $\XK(B)\cong\XK(B_1\oplus\Sigma B_2)$. By the definition of~$\XK$, the element~$f$ induces isomorphisms $\KK(X;i_x\C,B)\cong\KK(X;i_x\C,B_1\oplus\Sigma B_2)$ for all $x\in X$. The usual bootstrap argument shows that~$f$ induces isomorphisms $\KK(X;D,B)\cong\KK(X;D,B_1\oplus\Sigma B_2)$ for every object~$D$ in~$\Bootstrap(X)$. We may thus replace~$A$ by $A_1\oplus\Sigma A_2$ and~$B$ by $B_1\oplus\Sigma B_2$. Hence the sequence~\eqref{eq:uct} decomposes as a direct sum of four sequences in which, for degree reasons, either the left-hand or the right-hand term vanishes, making the construction of a splitting trivial.
\end{proof}

\begin{proposition}
	\label{pro:commutative}
Let~$A$ be a separable \Cstar{}algebra over~$X$ with vanishing boundary maps. Then there is a commutative \Cstar{}algebra~$C$ over~$X$ such that $\XK(A)\cong\XK(C)$. The spectrum of~$C$ may be chosen to be at most three-dimensional. If $\XK(A)$ is finitely generated, the spectrum of~$C$ may be chosen to be a finite complex of dimension at most three.
\end{proposition}

\begin{proof}
It is straight-forward to generalize the argument from \cite{Blackadar:Kbook}*{Corollary~23.10.3}. Using that modules split into even and odd part, a suspension argument reduces to the case that $\XK(A)$ vanishes in degree zero. Choose a projective resolution
\[
 0\to P_1\xrightarrow{f} P_0\to\XK(A)\to 0
\]
such that $P_i=\bigoplus_{x\in X}\bigoplus_\N (P^x\oplus P^x[1])$. Setting $D_i=\bigoplus_{x\in X}\bigoplus_\N\bigbr{i_x\C\otimes C(S^1)}$, we have $P_i\cong\XK(D_i)$. Then there is a \Star{}homomorphism $\varphi\colon D_1\to D_0$ over~$X$ inducing the map~$f$. The mapping cone of~$\varphi$ has the desired properties. In the finitely generated case, it clearly suffices to use finite direct sums instead of countable ones.
\end{proof}

\begin{corollary}
Let~$A$ be a separable \Cstar{}algebra over~$X$ with vanishing boundary maps. Then~$A$ belongs to the bootstrap class~$\Bootstrap(X)$ if and only if~$A$ is $\KK(X)$\nb-equiv\-a\-lent to a commutative \Cstar{}algebra over~$X$.
\end{corollary}

\begin{proof}
If~$A$ is a commutative \Cstar{}algebra over~$X$ then it is nuclear and the subquotient $A(\{x\})$ belongs to the bootstrap class~$\Bootstrap$ for every $x\in X$. Hence~$A$ belongs to $\Bootstrap(X)$ by \cite{MN:Bootstrap}*{Corollary~4.13}. Since~$\Bootstrap(X)$ is closed under $\KK(X)$\nb-equiv\-a\-lence, one implication follows. The converse implication follows from Proposition~\ref{pro:commutative} and Corollary~\ref{cor:lifting}.
\end{proof}

\begin{remark}
The stable homological functor~$\OK$ does not fit into this framework as nicely: if the space $X$ is sufficiently complicated then $\OK$ is not universal for its kernel on morphisms because it has ``hidden symmetries.'' More precisely, there are natural transformations among the $\K$-the\-o\-ret\-i\-cal functors comprised by the invariant~$\OK$, the action of which is not part of the definition of~$\OK$ (compare \cite{MN:Filtrated}*{\S2.1}).
\end{remark}

\section{Classification of certain Kirchberg \texorpdfstring{$X$}{X}-algebras}

In this section, we use our universal coefficient theorem to obtain classification results for Kirchberg $X$\nb-al\-ge\-bras with vanishing boundary maps.

\begin{definition}
A \Cstar{}algebra over~$X$ is a \emph{Kirchberg $X$\nb-algebra} if it is tight, nuclear, purely infinite and separable.
\end{definition}

\begin{theorem}
	\label{thm:classification_of_KXA_with_RR0_and_RP}
Let~$A$ and~$B$ be stable real-rank-zero Kirchberg $X$\nb-al\-ge\-bras with intermediate cancellation and simple subquotients in the bootstrap class~$\Bootstrap$. Then every isomorphism $\XK(A)\cong\XK(B)$ can be lifted to a \Star{}iso\-mor\-phism over~$X$. Consequently, every isomorphism $\OK(A)\cong\OK(B)$ can be lifted to a \Star{}iso\-mor\-phism over~$X$.
\end{theorem}

\begin{proof}
By Propositions~\ref{pro:characterization_of_KXA_with_vanishing_exponential_maps} and~\ref{pro:characterization_of_KXA_with_vanishing_index_maps}, the algebras $A$ and~$B$ have vanishing boundary maps. Hence the first claim follows from Corollary~\ref{cor:lifting} together with Kirchberg's classification theorem~\cite{Kirchberg}. Recall that a nuclear \Cstar{}algebra belongs to $\Bootstrap(X)$ if and only if the fibre $A(\{x\})$ belongs to~$\Bootstrap$ for every $x\in X$ by~\cite{MN:Bootstrap}*{Corollary~4.13}. Notice also that stable nuclear purely infinite \Cstar{}algebras with real rank zero are $\Cuntz_\infty$\nb-absorbing by \cite{Kirchberg_Rordam:Infinite}*{Corollary~9.4}. The second claim follows from the equivalence in Lemma~\ref{lem:res_and_colim_inverse}.
\end{proof}

Next, we establish a range result for the invariant~$\OK$ on stable real-rank-zero Kirchberg $X$-algebra with intermediate cancellation. For this, we need to assume that~$X$ is a unique path space.

\begin{theorem}
  \label{thm:range_on_KXA}
Assume that~$X$ is a unique path space. Let~$M$ be a flabby cosheaf on~$\Open(X)$. Then there is a stable real-rank-zero Kirchberg $X$-algebra with intermediate cancellation and simple subquotients in the bootstrap class~$\Bootstrap$ such that $\OK(A)\cong M$.
\end{theorem}

\begin{proof}
Since~$M$ is a flabby cosheaf, its restriction $\Res(M)\inOb\Reps(X)$ has a pro\-jec\-tive resolution of length~$1$ by Lemma~\ref{lem:proj_dim_1}. A simple argument as in \cite{MN:Filtrated}*{Theorem~4.8} shows that there is a separable \Cstar{}algebra~$A$ over~$X$ in the bootstrap class $\Bootstrap(X)$ with $\XK(A)\cong\Res(M)$. By \cite{MN:Bootstrap}*{Corollary~5.5}, we may assume that~$A$ is a stable Kirchberg $X$-algebra with simple subquotients in~$\Bootstrap$.

Since~$X$ is a unique path space, the set $U_x\setminus\{x\}$ is the disjoint union of the sets~$U_y$, where~$y$ is a closed point in $U_x\setminus\{x\}$. Hence the map $\K_*\bigbr{A(U_x\setminus\{x\})}\to\K_*\bigbr{A(U_x)}$ identifies with the map $M(U_x\setminus\{x\})\to M(U_x)$ because $\K$\nb-theory preserves direct sums and cosheaves take disjoint unions to direct sums. Since~$M$ is flabby by assumption, Corollary~\ref{cor:vanishing_boundary_maps_criterion} therefore shows that~$A$ has vanishing boundary maps. Thus~$A$ has real rank zero and intermediate cancellation by Propositions~\ref{pro:characterization_of_KXA_with_vanishing_exponential_maps} and~\ref{pro:characterization_of_KXA_with_vanishing_index_maps} and we have $\OK(A)\cong\Colim\bigbr{\XK(A)}\cong\Colim\bigbr{\Res(M)}\cong M$.
\end{proof}

\begin{corollary}
  \label{cor:complete_classification_of_KXA}
Assume that~$X$ is a unique path space. The functors $\OK$ and $\XK$ implement bijections of isomorphism classes of
\begin{itemize}
\item stable real-rank-zero Kirchberg $X$\nb-al\-ge\-bras with intermediate cancellation and simple subquotients in the bootstrap class~$\Bootstrap$,
\item flabby cosheaves on $\Open(X)$,
\item representations of~$X$ whose associated cosheaf is flabby.
\end{itemize}
\end{corollary}

\begin{proof}
Denote the three sets above by (Kirchberg), (Cosheaves) and (Representations), respectively. We have maps induced by functors as indicated in the following commutative diagram.
\[
 \xymatrix{
& \textup{(Kirchberg)}\ar@/_1pc/[dl]_{\OK}\ar@/^1pc/[dr]^{\XK} & \\
\textup{(Cosheaves)}\ar@/_1pc/[rr]^{\Res} & & \textup{(Representations)}\ar@/_1pc/[ll]^{\Colim}
}
\]
We observed in \S\ref{sec:reps_and_cosheaves} that the functors $\Res$ and $\Colim$ induce mutually inverse bijections. By Theorem~\ref{thm:classification_of_KXA_with_RR0_and_RP}, the functor~$\XK$ induces an injective map. By Theorem~\ref{thm:range_on_KXA}, the functor~$\OK$ induces a surjective map. Hence all four maps are bijective.
\end{proof}

Now we enhance our invariant in order to obtain a classification result in the unital case.

\begin{definition}
A \emph{pointed cosheaf} on~$\Open(X)$ is a cosheaf~$M$ on~$\Open(X)$ together with a distinguished element $m\in M(X)_0$. A morphism of pointed cosheaves is a morphism of cosheaves preserving the distinguished element. The category of pointed cosheaves on~$\Open(X)$ is denoted by $\CoSheaf\bigbr{\Open(X)}^+$.
\end{definition}

\begin{definition}
Let $\KKcat(X)^+$ denote the full subcategory of $\KKcat(X)$ consisting of all unital separable \Cstar{}algebras over~$X$. We define a functor $\OK^+\colon\KKcat(X)^+\to\CoSheaf\bigbr{\Open(X)}^+$ by
\[
\OK^+(A) = \bigbr{\OK(A),[1_A]}.
\]
\end{definition}

\begin{corollary}
  \label{cor:unital_classification_of_KXA}
Let~$A$ and~$B$ be unital real-rank-zero Kirchberg $X$\nb-al\-ge\-bras with intermediate cancellation and simple subquotients in the bootstrap class~$\Bootstrap$.
Then every isomorphism $\OK^+(A)\cong\OK^+(B)$ can be lifted to a \Star{}iso\-mor\-phism over~$X$.
\end{corollary}

\begin{proof}
This is a consequence of the strong stable classification result in Theorem~\ref{thm:classification_of_KXA_with_RR0_and_RP} using the so-called meta theorem \cite{Eilers-Restorff-Ruiz:generalized_meta_theorem}*{Theorem~3.3}.
\end{proof}

\section{Cosheaves arising as invariants of graph \texorpdfstring{$C^*$}{C*}-algebras}
	\label{sec:range_on_graph_algebras}

In this section, we provide range results for the invariants~$\OK$ and~$\OK^+$ on purely infinite tight graph \Cstar{}algebra over~$X$ with intermediate cancellation. For definitions and general facts concerning graph \Cstar{}al\-ge\-bras we refer to~\cite{Raeburn}. The Cuntz--Krieger algebras introduced in \cites{Cuntz--Krieger:A_class,Cuntz:A_class_II} are in particular unital graph \Cstar{}al\-ge\-bras; when using the word Cuntz--Krieger algebra we implicitly assume that the underlying square matrix satisfies Cuntz's condition~(II), which ensures that the algebra is purely infinite.

\begin{definition}
A \emph{tight graph \Cstar{}algebra over~$X$} is a graph \Cstar{}algebra $C^*(E)$ equipped with a homeomorphism $\Prim\bigbr{C^*(E)}\to X$. A \emph{tight Cuntz--Krieger algebra over~$X$} is defined analogously.
\end{definition}

We point out that a purely infinite tight graph \Cstar{}algebra over~$X$ is in particular a real-rank-zero Kirchberg $X$\nb-al\-ge\-bras with simple subquotients in the bootstrap class~$\Bootstrap$ (see \cite{Raeburn}*{Remark 4.3} and \cite{hongszymanski}*{\S2}). Hence the classification results in the previous section apply to purely infinite tight graph \Cstar{}algebras over~$X$ with intermediate cancellation. We obtain the following corollary.

\begin{corollary}
Let~$A$ and~$B$ be purely infinite tight graph \Cstar{}al\-ge\-bras over~$X$ with intermediate cancellation. If $\OK(A)\cong\OK(B)$, then~$A$ is stably isomorphic to~$B$. If~$A$ and~$B$ are unital and $\OK^+(A)\cong\OK^+(B)$, then~$A$ is isomorphic to~$B$.
\end{corollary}

It is now natural to ask which (pointed) cosheaves arise as the invariant of a (unital) purely infinite tight graph \Cstar{}al\-ge\-bra over~$X$ with intermediate cancellation.

\begin{definition}
  \label{def:properties_of_flabby_cosheaves}
A flabby cosheaf~$M$ on $\Open(X)$ is said to have \emph{free quotients in odd degree} if the quotient $M(V)_1/M(U)_1$ is free for all open subsets $U\subseteq V\subseteq X$. We say that~$M$ has \emph{finite ordered ranks} if, for all $U\in\Open(X)$,
\[
 \rank M(U)_1\leq\rank M(U)_0 < \infty.
\]
Similarly, we say that~$M$ has \emph{finite equal ranks} if
$\rank M(U)_1 = \rank M(U)_0 < \infty$
for all $U\in\Open(X)$.
A pointed cosheaf is called \emph{flabby} if the underlying cosheaf is flabby. A flabby pointed cosheaf has one of the three properties above if this is the case for the underlying cosheaf.
\end{definition}

We will use the invariant $\FK_\Catgunnar$ for \Cstar{}algebras over~$X$ from~\cite{Arklint-Bentmann-Katsura:Reduction}.

\begin{definition}[\cite{Arklint-Bentmann-Katsura:Reduction}*{Definition~6.1}]
An $\Catgunnar$-module~$N$ is a collection of Abelian groups $N(\{x\})_1$, $N(U_x)_0$ and $N(U_x\setminus\{x\})_0$ for $x\in X$ together with group homomorphisms $\delta_{\{x\}}^{U_x\setminus\{x\}}\colon N(\{x\})_1\to N(U_x\setminus\{x\})_0$ and $i_{U_x\setminus\{x\}}^{U_x}\colon N(U_x\setminus\{x\})_0\to N(U_x)_0$ for $x\in X$ and $i_{U_y}^{U_x\setminus\{x\}}\colon N(U_y)_0\to N(U_x\setminus\{x\})_0$ for all pairs $(x,y)$ with $y\to x$ such that certain relations are fulfilled. A homomorphism of $\Catgunnar$-modules is a collection of group homomorphisms making all squares commute.
\end{definition}

There is a notion of \emph{exactness} for $\Catgunnar$-modules (see \cite{Arklint-Bentmann-Katsura:Reduction}*{Definition~6.5}) and our notation suggests an obvious $\K$\nb-the\-o\-ret\-i\-cal functor $\FK_\Catgunnar$ from $\KKcat(X)$ to exact $\Catgunnar$-modules (see \cite{Arklint-Bentmann-Katsura:Reduction}*{Definition~6.4 and Corollary~6.9}). Notice that, for $U_x\subseteq U_y$, we can obtain the map $\K_0\bigbr{A(U_x)}\to\K_0\bigbr{A(U_y)}$ by composing maps that are part of the invariant~$\FK_\Catgunnar(A)$.

\begin{theorem}
  \label{thm:stable_classification_of_graph_algs}
A flabby cosheaf on~$\Open(X)$ is isomorphic to $\OK\bigbr{C^*(E)}$ for some purely infinite tight graph \Cstar{}algebra $C^*(E)$ over~$X$ with intermediate cancellation if and only if it has free quotients in odd degree.
\end{theorem}

\begin{proof}
It is well-known that graph \Cstar{}algebras have free $\K_1$\nb-groups. Since (gauge-invariant) ideals in graph \Cstar{}algebras are themselves graph \Cstar{}algebras by~\cite{Ruiz-Tomforde:Ideals}, it follows that $\OK\bigbr{C^*(E)}$ has free quotients in odd degree if $C^*(E)$ is a purely infinite tight graph \Cstar{}algebra over~$X$.

Conversely, let~$M$ be a flabby cosheaf on~$\Open(X)$ that has free quotients in odd degree. We associate to~$M$ an $\Catgunnar$-module~$N$ in the following way: for $x\in X$, set $N(U_x)_0=M(U_x)_0$, $N(U_x\setminus\{x\})_0=M(U_x\setminus\{x\})_0$ and let $N(\{x\})_1$ be the quotient of $M(U_x)_1$ by $M(U_x\setminus\{x\})_1$. The maps $i_{U_x\setminus\{x\}}^{U_x}$ and $i_{U_y}^{U_x\setminus\{x\}}$ for~$N$ are defined to be the even parts of the identically denoted maps for~$M$. The homomorphisms $\delta_{\{x\}}^{U_x\setminus\{x\}}$ are defined to be the zero homomorphisms.

To check that this really defines an $\Catgunnar$-module, one has to verify the relations (6.2) and (6.3) in~\cite{Arklint-Bentmann-Katsura:Reduction}. This is straight-forward: the relation (6.2) is fulfilled because we have defined the maps $\delta_{\{x\}}^{U_x\setminus\{x\}}$ as zero maps; the relation (6.3) follows from the fact that the composition $M(U)\xrightarrow{i_U^V} M(V)\xrightarrow{i_V^W} M(W)$ is equal to $M(U)\xrightarrow{i_U^W} M(W)$ for all open subsets $U\subseteq V\subseteq W\subseteq X$. We observe that the $\Catgunnar$-module~$N$ is exact: the exactness of the sequence (6.7) in~\cite{Arklint-Bentmann-Katsura:Reduction} follows from the fact that~$M$ is a cosheaf; the sequence (6.6) in~\cite{Arklint-Bentmann-Katsura:Reduction} is exact because~$M$ is flabby.

Since~$M$ has free quotients in odd degree, \cite{Arklint-Bentmann-Katsura:Reduction}*{Theorem~8.2} implies that there is a purely infinite tight graph \Cstar{}algebra $C^*(E)$ over~$X$ such that $\FK_\Catgunnar\bigbr{C^*(E)}\cong N$. Since~$C^*(E)$ has real rank zero, it has vanishing exponential maps, so that the $\K_0$-groups of its ideals form an (ungraded) cosheaf on~$\Open(X)$. This cosheaf coincides with the even part of~$M$ on the basis of minimal open neighbourhoods of points. Since cosheaves are determined by their restriction to a basis, the (ungraded) cosheaves~$M_0$ and $\OK\bigbr{C^*(E)}_0$ are isomorphic. Since~$M$ is flabby this shows that $C^*(E)$ has vanishing index maps and therefore intermediate cancellation.

Exploiting freeness of the $\K_1$-groups and vanishing of boundary maps, we obtain isomorphisms
\[
\K_1\bigbr{C^*(E)(U)}\cong\bigoplus_{x\in U}\K_1\bigbr{C^*(E)(\{x\})}
\]
for all open subsets $U\subseteq X$ such that, under this identification, the homomorphisms induced by the ideal inclusions correspond to the obvious subgroup inclusions. Analogously, we have isomorphisms $M(U)_1\cong\bigoplus_{x\in U} N(\{x\})_1$ for all open subsets $U\subseteq X$ because~$M$ is flabby and has free quotients in odd degree. Hence $\OK\bigbr{C^*(E)}_1\cong M_1$. It follows that $\OK\bigbr{C^*(E)}\cong M$ as desired.
\end{proof}

For the proof of the next result, we need to recall that there is a notion of \emph{pointed} $\Catgunnar$-module (see \cite{Arklint-Bentmann-Katsura:Reduction}*{Definition~9.1}) and a functor $\FK_\Catgunnar^+$ from $\KKcat(X)^+$ to pointed $\Catgunnar$-modules.

\begin{theorem}
  \label{thm:unital_classification_of_graph_algs}
A flabby pointed cosheaf on~$\Open(X)$ is isomorphic to $\OK^+\bigbr{C^*(E)}$ for some unital purely infinite tight graph \Cstar{}algebra $C^*(E)$ over~$X$ with intermediate cancellation if and only if it has free quotients in odd degree and finite ordered ranks.
\end{theorem}

\begin{proof}
Again, the well-known formulas for the $\K$\nb-theory of graph \Cstar{}algebras show that $\OK\bigbr{C^*(E)}$ has free quotients in odd degree and finite ordered ranks if $C^*(E)$ is a unital purely infinite tight graph \Cstar{}algebra over~$X$. Conversely, to a given flabby pointed cosheaf~$(M,m)$ we associate an exact pointed $\Catgunnar$-module~$(N,n)$ as in the previous proof. Our assumptions on~$M$ then guarantee that we can apply \cite{Arklint-Bentmann-Katsura:Reduction}*{Theorem~9.11} to obtain a unital purely infinite tight graph \Cstar{}algebra $C^*(E)$ over~$X$ such that there is an isomorphism of pointed $\Catgunnar$-modules $\FK_\Catgunnar^+\bigbr{C^*(E)}\cong (N,n)$. An argument as in the previous proof shows that $\OK^+\bigbr{C^*(E)}\cong (M,m)$ and that $C^*(E)$ has intermediate cancellation.
\end{proof}

\begin{theorem}
  \label{thm:unital_classification_of_CK_algs}
A flabby pointed cosheaf on~$\Open(X)$ is isomorphic to $\OK^+(\Cuntz_A)$ for some tight Cuntz--Krieger algebra $\Cuntz_A$ over~$X$ with intermediate cancellation if and only if it has free quotients in odd degree and finite equal ranks.
\end{theorem}

\begin{proof}
The $\K$\nb-theory formulas for graph \Cstar{}algebras imply that the cosheaf $\OK(\Cuntz_A)$ has finite equal degrees if~$\Cuntz_A$ is a tight Cuntz--Krieger algebra over~$X$. Conversely, $\OK(\Cuntz_A)$ having finite equal ranks implies that $\FK_\Catgunnar(A)$ meets the additional conditions in \cite{Arklint-Bentmann-Katsura:Reduction}*{Theorem~8.2} that guarantee that the graph~$E$ in the previous proof can be chosen finite (it has no sinks or sources by construction).
\end{proof}

\section{Extensions of Cuntz--Krieger algebras}

In this section, we establish a permanence property of Cuntz--Krieger algebras with intermediate cancellation with respect to extensions.

\begin{definition}[\cite{Arklint:PhantomCK}*{Definition~1.1}]
	\label{def:look_like_a_CKA}
A \Cstar{}algebra~$A$ over~$X$ \emph{looks like a Cuntz--Krieger algebra} if~$A$ is a unital real-rank-zero Kirchberg $X$-algebra with simple subquotients in the bootstrap class~$\Bootstrap$ such that, for all $x\in X$, the group $\K_1\bigbr{A(\{x\})}$ is free and $\rank\K_0\bigbr{A(\{x\})}=\rank\K_1\bigbr{A(\{x\})}<\infty$.

A \Cstar{}algebra~$A$ over~$X$ that satisfies these conditions but is stable rather than unital is said to \emph{look like a stabilized Cuntz--Krieger algebra}.
\end{definition}

The following result generalizes the observation in~\cite{Arklint:PhantomCK}*{Corollary 2.4}, which is concerned with Cuntz--Krieger algebras with trivial $\K$\nb-theory.

\begin{corollary}
	\label{cor:phantoms_with_intermediate_cancellation}
Let~$A$ be a \Cstar{}algebra over~$X$ that looks like a Cuntz--Krieger algebra and has intermediate cancellation. Then~$A$ is \Star{}isomorphic over~$X$ to a tight Cuntz--Krieger algebra over~$X$ with intermediate cancellation.
\end{corollary}

\begin{proof}
Let~$B$ be a \Cstar{}algebra over~$X$ with intermediate cancellation that looks like a Cuntz--Krieger algebra. Repeated use of the six-term exact sequence shows that $\OK(B)$ has free quotients in odd degree and finite equal ranks. By Theorem~\ref{thm:unital_classification_of_CK_algs}, there is a tight Cuntz--Krieger algebra $\Cuntz_A$ over~$X$ with intermediate cancellation such that $\OK^+(B)\cong\OK^+(\Cuntz_A)$. By Corollary~\ref{cor:unital_classification_of_KXA}, we have $B\cong\Cuntz_A$.
\end{proof}

\begin{corollary}
	\label{cor:stable_phantoms_with_intermediate_cancellation}
Let~$A$ be a \Cstar{}algebra over~$X$ that looks like a stabilized Cuntz--Krieger algebra and has intermediate cancellation. Then~$A$ is stably isomorphic over~$X$ to a tight Cuntz--Krieger algebra over~$X$ with intermediate cancellation.
\end{corollary}

\begin{proof}
Let~$B$ be a \Cstar{}algebra over~$X$ with intermediate cancellation that looks like a stabilized Cuntz--Krieger algebra. As in the previous proof, we see that $\OK(B)$ has free quotients in odd degree and finite equal ranks. We turn the cosheaf $\OK(B)$ into a pointed cosheaf by choosing an arbitrary element in $\K_0(B)$. By Theorem~\ref{thm:unital_classification_of_CK_algs}, there is a tight Cuntz--Krieger algebra~$\Cuntz_A$ over~$X$ with intermediate cancellation such that $\OK(B)\cong\OK(\Cuntz_A)$. By Theorem~\ref{thm:classification_of_KXA_with_RR0_and_RP}, the algebras~$B$ and~$\Cuntz_A$ are stably isomorphic over~$X$.
\end{proof}

\begin{theorem}
	\label{thm:extensions}
Let $I\rightarrowtail A\twoheadrightarrow B$ be an extension of \Cstar{}algebras. Assume that~$A$ is unital. Then~$A$ is a Cuntz--Krieger algebra with intermediate cancellation if and only if
\begin{itemize}
\item the ideal $I$ is stably isomorphic to a Cuntz--Krieger algebra with intermediate cancellation,
\item the quotient $B$ is a Cuntz--Krieger algebra with intermediate cancellation,
\item the boundary map $\K_*(B)\to\K_{*+1}(I)$ vanishes.
\end{itemize}
A similar assertion holds for extensions of unital purely infinite graph \Cstar{}algebras with intermediate cancellation.
\end{theorem}

\begin{proof}
The crucial point is that the property of \emph{looking like} a Cuntz--Krieger algebra behaves well with extensions (see Remark \ref{rem:rrzero}). So does intermediate cancellation when considered for separable purely infinite \Cstar{}algebras by Corollary~\ref{cor:intermediate_cancellation_and_extensions}. We have that~$A\inOb\KKcat\bigbr{\Prim(A)}$ looks like a Cuntz--Krieger algebra and has intermediate cancellation if and only if
\begin{itemize}
\item the stabilization $I\otimes\Compacts\inOb\Cstarcat\bigbr{\Prim(I)}$ of the ideal~$I$ looks like a stabilized Cuntz--Krieger algebra and has intermediate cancellation,
\item the quotient $B\inOb\Cstarcat\bigbr{\Prim(B)}$ looks like a Cuntz--Krieger algebra and has intermediate cancellation,
\item the boundary map $\K_*(B)\to\K_{*+1}(I)$ vanishes.
\end{itemize}
Hence the result follows from Corollary~\ref{cor:phantoms_with_intermediate_cancellation} applied to~$A$ and~$B$ and from Corollary~\ref{cor:stable_phantoms_with_intermediate_cancellation} applied to~$I$. The assertion for unital graph \Cstar{}algebras follows similarly from Theorem~\ref{thm:unital_classification_of_graph_algs} and Corollary~\ref{cor:unital_classification_of_KXA}.
\end{proof}

As similar argument based on Theorems~\ref{thm:stable_classification_of_graph_algs} and~\ref{thm:classification_of_KXA_with_RR0_and_RP} leads to the following permanence result for stabilized purely infinite graph \Cstar{}algebras.

\begin{theorem}
Let $I\rightarrowtail A\twoheadrightarrow B$ be an extension of \Cstar{}algebras. Assume that~$A$ has finite ideal lattice. Then~$A$ is stably isomorphic to a purely infinite graph \Cstar{}algebra with intermediate cancellation if and only if
\begin{itemize}
\item the ideal $I$ is stably isomorphic to a purely infinite graph \Cstar{}algebra with intermediate cancellation,
\item the quotient $B$ is stably isomorphic to a purely infinite graph \Cstar{}algebra with intermediate cancellation,
\item the boundary map $\K_*(B)\to\K_{*+1}(I)$ vanishes.
\end{itemize}
\end{theorem}

\section*{Acknowledgement}
I would like to thank Lawrence G.\ Brown and Mikael \Rordam{} for helpful correspondence, and James Gabe and Kristian Moi for useful discussions.


\begin{bibsection}
  \begin{biblist}
  
\bib{Arklint:PhantomCK}{article}{
   author={Arklint, Sara},
   title={Do phantom Cuntz--Krieger algebras exist?},
    eprint = {arXiv:1210.6515},
      year = {2012},
 }
 
\bib{Arklint-Bentmann-Katsura:Reduction}{article}{
   author={Arklint, Sara},
   author={Bentmann, Rasmus},
   author={Katsura, Takeshi},  
   title={Reduction of filtered $\K$-theory and a characterization of Cuntz--Krieger algebras},
    eprint = {arXiv:1301.7223},
      year = {2013},
 }
 
\bib{Bentmann:Thesis}{article}{
      author={Bentmann, Rasmus},
       title={Filtrated {K}-theory and classification of {$C^*$}-algebras},
        date={University of {G}\"ottingen, 2010},
        note={Diplom thesis, available online at: \href{www.math.ku.dk/~bentmann/thesis.pdf}{www.math.ku.dk/\textasciitilde bentmann/thesis.pdf}},
}
 
\bib{BK}{article}{
  author={Bentmann, Rasmus},
  author={K\"ohler, Manuel},
  title={Universal Coefficient Theorems for $C^*$-algebras over finite topological spaces},
  eprint = {arXiv:math/1101.5702},
  year = {2011},
}

\bib{Bentmann-Meyer}{article}{
  author={Bentmann, Rasmus},
  author={Meyer, Ralf},
  title={Circle actions on $C^*$-algebras up to $\KK$-equivalence},
  note = {in preparation},
}

\bib{Blackadar:Kbook}{book}{
   author={Blackadar, Bruce},
   title={$K$-theory for operator algebras},
   series={Mathematical Sciences Research Institute Publications},
   volume={5},
   edition={2},
   publisher={Cambridge University Press},
   place={Cambridge},
   date={1998},
   pages={xx+300},
   isbn={0-521-63532-2},
   review={\MRref{1656031}{99g:46104}},
}

\bib{Bonkat:Thesis}{thesis}{
  author={Bonkat, Alexander},
  title={Bivariante \(K\)\nobreakdash -Theorie f\"ur Kategorien projektiver Systeme von \(C^*\)\nobreakdash -Al\-ge\-bren},
  date={2002},
  institution={Westf. Wilhelms-Universit\"at M\"unster},
  type={phdthesis},
  language={German},
  note={Available at the Deutsche Nationalbibliothek at \url {http://deposit.ddb.de/cgi-bin/dokserv?idn=967387191}},
}

\bib{Bredon:Cosheaves}{article}{
   author={Bredon, Glen E.},
   title={Cosheaves and homology},
   journal={Pacific J. Math.},
   volume={25},
   date={1968},
   pages={1--32},
   issn={0030-8730},
   review={\MRref{0226631}{37 \#2220}},
}

\bib{Brown:intermediate_cancellation}{article}{
   author={Brown, Lawrence G.},
   date={2013},
   note={personal communication},
}

\bib{Brown-Pedersen:RR0}{article}{
   author={Brown, Lawrence G.},
   author={Pedersen, Gert K.},
   title={$C^*$-algebras of real rank zero},
   journal={J. Funct. Anal.},
   volume={99},
   date={1991},
   number={1},
   pages={131--149},
   issn={0022-1236},
   review={\MRref{1120918}{(92m:46086)}},
   doi={10.1016/0022-1236(91)90056-B},
}

\bib{Brown-Pedersen:Non-stable}{article}{
   author={Brown, Lawrence G.},
   author={Pedersen, Gert K.},
   title={Non-stable $\K$-theory and entremally rich $C^*$-algebras},
   eprint = {arXiv:math/0708.3078},
   year={2007},
}

\bib{Cuntz--Krieger:A_class}{article}{
   author={Cuntz, Joachim},
   author={Krieger, Wolfgang},
   title={A class of $C^{\ast} $-algebras and topological Markov chains},
   journal={Invent. Math.},
   volume={56},
   date={1980},
   number={3},
   pages={251--268},
   issn={0020-9910},
   review={\MRref{561974}{ (82f:46073a)}},
   doi={10.1007/BF01390048},
}

\bib{Cuntz:A_class_II}{article}{
   author={Cuntz, J.},
   title={A class of $C^{\ast} $-algebras and topological Markov chains.
   II. Reducible chains and the Ext-functor for $C^{\ast} $-algebras},
   journal={Invent. Math.},
   volume={63},
   date={1981},
   number={1},
   pages={25--40},
   issn={0020-9910},
   review={\MRref{608527}{ (82f:46073b)}},
   doi={10.1007/BF01389192},
}

\bib{Eilers-Restorff-Ruiz:generalized_meta_theorem}{article}{
  author={Eilers, S\o ren},
  author={Restorff, Gunnar},
  author={Ruiz, Efren},
  title={Strong classification of extensions of classifiable $C^*$-algebras },
   eprint = {arXiv:math/arXiv:1301.7695},
   year={2013},
}

\bib{hongszymanski}{article}{
    AUTHOR = {Hong, Jeong Hee},
    author = {Szyma{\'n}ski, Wojciech},
     TITLE = {Purely infinite {C}untz-{K}rieger algebras of directed graphs},
   JOURNAL = {Bull. London Math. Soc.},
    VOLUME = {35},
      YEAR = {2003},
    NUMBER = {5},
     PAGES = {689--696},
      ISSN = {0024-6093},
       DOI = {10.1112/S0024609303002364},
       URL = {http://dx.doi.org/10.1112/S0024609303002364},
}

\bib{Kirchberg}{article}{
  author={Kirchberg, Eberhard},
  title={Das nicht-kommutative Michael-Auswahlprinzip und die Klassifikation nicht-einfacher Algebren},
  language={German, with English summary},
  conference={ title={$C^*$-algebras}, address={M\"unster}, date={1999}, },
  book={ publisher={Springer}, place={Berlin}, },
  date={2000},
  pages={92--141},
  review={\MRref {1796912}{2001m:46161}},
}


\bib{Kirchberg_Rordam:Non-simple}{article}{
   author={Kirchberg, Eberhard},
   author={R{\o}rdam, Mikael},
   title={Non-simple purely infinite $C^\ast$-algebras},
   journal={Amer. J. Math.},
   volume={122},
   date={2000},
   number={3},
   pages={637--666},
   issn={0002-9327},
   review={\MRref{1759891}{ (2001k:46088)}},
}

\bib{Kirchberg_Rordam:Infinite}{article}{
   author={Kirchberg, Eberhard},
   author={R{\o}rdam, Mikael},
   title={Infinite non-simple $C^*$-algebras: absorbing the Cuntz
   algebras $\mathcal O_\infty$},
   journal={Adv. Math.},
   volume={167},
   date={2002},
   number={2},
   pages={195--264},
   issn={0001-8708},
   review={\MRref{1906257}{(2003k:46080)}},
   doi={10.1006/aima.2001.2041},
}

\bib{Lin-Rordam:Extensions_of_inductive_limits}{article}{
   author={Lin, Hua Xin},
   author={R{\o}rdam, Mikael},
   title={Extensions of inductive limits of circle algebras},
   journal={J. London Math. Soc. (2)},
   volume={51},
   date={1995},
   number={3},
   pages={603--613},
   issn={0024-6107},
   review={\MRref{1332895}{ (96m:46104)}},
   doi={10.1112/jlms/51.3.603},
}
 
\bib{Meyer-Nest:Homology_in_KK}{article}{
   author={Meyer, Ralf},
   author={Nest, Ryszard},
   title={Homological algebra in bivariant $K$-theory and other triangulated
   categories. I},
   conference={
      title={Triangulated categories},
   },
   book={
      series={London Math. Soc. Lecture Note Ser.},
      volume={375},
      publisher={Cambridge Univ. Press},
      place={Cambridge},
   },
   date={2010},
   pages={236--289},
   review={\MRref{2681710 }{(2012f:19015)}},
}
 
\bib{MN:Filtrated}{article}{
   author={Meyer, Ralf},
   author={Nest, Ryszard},
   title={${\rm C}^*$-algebras over topological spaces: filtrated
   $\K$-theory},
   journal={Canad. J. Math.},
   volume={64},
   date={2012},
   number={2},
   pages={368--408},
   issn={0008-414X},
   review={\MRref{2953205}{}},
   doi={10.4153/CJM-2011-061-x},
}

\bib{MN:Bootstrap}{article}{
  author={Meyer, Ralf},
  author={Nest, Ryszard},
  title={$C^*$-algebras over topological spaces: the bootstrap class},
  journal={M\"unster J. Math.},
  volume={2},
  date={2009},
  pages={215--252},
  issn={1867-5778},
  review={\MRref {2545613}{}},
 }
 
 \bib{Raeburn}{book}{
   author={Raeburn, Iain},
   title={Graph algebras},
   series={CBMS Regional Conference Series in Mathematics},
   volume={103},
   publisher={Published for the Conference Board of the Mathematical
   Sciences, Washington, DC},
   year={2005},
   pages={vi+113},
   isbn={0-8218-3660-9},
   review={\MRref{2135030}{2005k:46141}},
}

\bib{Restorff:Thesis}{thesis}{
  author={Restorff, Gunnar},
  title={Classification of Non-Simple $\textup C^*$\nobreakdash -Algebras},
  type={phdthesis},
  institution={K{\o }benhavns Universitet},
  date={2008},
  isbn={978-87-91927-25-6},
  eprint={http://www.math.ku.dk/~restorff/papers/afhandling_med_ISBN.pdf},
}

\bib{Rordam:Classification_of_nuclear}{article}{
   author={R{\o}rdam, M.},
   title={Classification of nuclear, simple $C^*$-algebras},
   conference={
      title={Classification of nuclear $C^*$-algebras. Entropy in
      operator algebras},
   },
   book={
      series={Encyclopaedia Math. Sci.},
      volume={126},
      publisher={Springer},
      place={Berlin},
   },
   date={2002},
   pages={1--145},
   review={\MRref{1878882}{(2003i:46060)}},
}

\bib{Rordam_Pasnicu:PI}{article}{
   author={Pasnicu, Cornel},
   author={R{\o}rdam, Mikael},
   title={Purely infinite $C^*$-algebras of real rank zero},
   journal={J. Reine Angew. Math.},
   volume={613},
   date={2007},
   pages={51--73},
   issn={0075-4102},
   review={\MRref{2377129}{(2009b:46119)}},
   doi={10.1515/CRELLE.2007.091},
}

 \bib{RS}{article}{
   author={Rosenberg, Jonathan},
   author={Schochet, Claude},
   title={The K\"unneth theorem and the universal coefficient theorem for Kasparov's generalized $\K$-functor},
   journal={Duke Math. J.},
   volume={55},
   date={1987},
   number={2},
   pages={431--474},
   issn={0012-7094},
   review={\MRref {894590}{88i:46091}},
   doi={10.1215/S0012-7094-87-05524-4},
 }

 \bib{Ruiz-Tomforde:Ideals}{article}{
   author={Ruiz, Efren},
   author={Tomforde, Mark},
   title={Ideals in Graph Algebras},
    eprint = {arXiv:math/1205.1247},
      year = {2012},
 }
 
 \bib{stacks-project}{book}{
  author={The Stacks Project Authors},
  title={Stacks Project},
  note={available at: \href{http://stacks.math.columbia.edu/}{http://stacks.math.columbia.edu/}},
}

  \end{biblist}
\end{bibsection}
\end{document}